\documentclass[11pt]{amsart}
\usepackage{graphics,verbatim,xcolor}

\usepackage[all]{xy}

\newtheorem{theorem}{Theorem}
\newtheorem{theorem*}{Theorem}
\newtheorem{proposition}{Proposition}
\newtheorem{lemma}{Lemma}  
 
\newtheorem{definition}{Definition}

\def\metric#1{\langle #1 \rangle}

\def\metrictwo#1{\langle\!\langle #1 \rangle\!\rangle}
\def\cmetrictwo#1{\langle\!\langle #1 \rangle\!\rangle_{\mathbb{C}}}

\def\normtwo#1{\lVert #1\rVert}

\def\R{\mathbb{R}}

\def\C{\mathbb{C}}

\def\CP{\mathbf{CP}}
\def\S{\mathbb{S}}
\def\l{\lambda}
\def\a{\alpha}

\def\e{\epsilon}

\def\g{\gamma}

\def\s{\sigma}

\def\t{\theta}
\def\r{\rho}

\def\w{\omega}
\def\x{\xi}
\def\y{\eta}
\def\z{\zeta}

\def\E{\mathbb{E}}

\def\cW{{\mathcal W}}

\def\cK{{\mathcal K}}
\def\cU{{\mathcal U}}

\def\into{\rightarrow}

\DeclareMathOperator{\diag}{diag}

\DeclareMathOperator{\re}{re}
\DeclareMathOperator{\im}{im}

\def\smallskip{\par\vspace{1mm}}
\def\medskip{\par\vspace{2mm}}
\def\bigskip{\par\vspace{3mm}}

\def\fr#1#2{\frac{#1}{#2}}

\def\m#1{\begin{bmatrix}#1\end{bmatrix}}

\newcount\equationcount
\def\thenumber{0}
\def\eq#1{\global\advance\equationcount by 1
   \def\thenumber{\number\equationcount}
                        {$$#1\eqno(\thenumber)$$}}

\def\CP{\mathbb{CP}}

\newcommand{\beq}{\begin{equation}}
\newcommand{\Leq}[1]{\label{#1}\end{equation}}

\begin{document}

\title[No Infinite Spin]{No infinite spin for planar total collision}
\author{Richard Moeckel}
\author{Richard Montgomery}
\address{School of Mathematics\\ University of Minnesota\\ Minneapolis MN 55455}

\email{rick@math.umn.edu}

\address{Dept. of Mathematics\\ University of California, Santa Cruz\\ Santa Cruz CA}

\email{rmont@count.ucsc.edu}



\begin{abstract}
The infinite spin problem concerns  the rotational behavior of total collision orbits in the $n$-body problem.    It has long been known that when a solution tends to  total collision then its  normalized configuration curve  must converge to the set of 
normalized central configurations.  In the planar n-body problem every normalized configuration determines a circle of rotationally equivalent normalized  configurations and, in particular, there are circles of normalized central configurations.  It's conceivable that by means of an {\em infinite spin}, a total collision solution could converge to such a circle instead of to a particular point on it.   Here we prove that this is not possible, at least  if the limiting circle of central configurations is isolated from other circles of central configurations.  (It is believed that all central configurations are isolated, but this is not known in general.)  Our  proof relies on  combining  the center manifold theorem  with the \L{}ojaseiwicz gradient inequality.  
\end{abstract}

\date{January 30, 2023}
\maketitle

\section{Introduction}
Consider the planar $n$ body problem with masses $m_i>0$, positions $q_i\in\R^2$ and let $q=(q_1,\ldots,q_n)\in \R^{2n}$   The motion is governed by 
Newton's equations of motion
$$m_i\ddot q_i = \nabla_i U(q)$$
where
\begin{equation}\label{eq_U}
U(q)= \sum_{i<j}\frac{m_im_j}{r_{ij}}\qquad r_{ij}=|q_i-q_j|
\end{equation}
and $\nabla_i$ is the partial gradient with respect to $q_i$.
The translation symmetry of the problem implies that we may assume without loss of generality that the total momentum is zero and that the center of mass is fixed at the origin, that is
\begin{equation}\label{eq_cofm}
m_1q_1+\ldots+m_nq_n=0\qquad m_1v_1+\ldots+m_nv_n=0.
\end{equation}

A solution $q(t)$ has a {\em total collision} at time $T$ if all of the positions $q_i(t)$ converge to the same point as $t\into T$. This collision point must be the center of mass so we have $q(t)\into 0\in \R^{2n}$.  Let 
$$I(q)= \sum_i m_i|q_i|^2$$
be the moment of inertia with respect to the origin.  We can think of $I(q)$ as the square of a {\em mass norm} on $\R^{2n}$ and write $I(q) = \normtwo{q}^2$.
Then the quantity $r(q)=\sqrt{I(q)}= \normtwo{q}$ is a convenient measure of the distance to collision.   The unit vector $\hat q = q/\normtwo{q}$ is the corresponding {\em normalized configuration}.    Taking into account the center of mass condition, the normalized configurations form a sphere $S\simeq \S^{2n-3}$.

A classical result of Chazy \cite{Chazy} about total collision solutions asserts that their normalized configuration curve
 $\hat q(t)$ converges to the set of normalized  central configurations as $t \to T$, the total collision time. 
\begin{definition}
A point $q\in\R^{2n}$ is a {\em central configuration or CC} if
\begin{equation}\label{eq_CC}
\nabla_i U(q) + \l m_iq_i=0
\end{equation}
for some constant $\l>0$.
\end{definition}
\noindent This means that the gravitational acceleration on the $i$-th body is $-\l q_i$ which points toward the origin and is proportional to the distance to the origin.  If we release the masses at a central configuration $q_0$ with initial zero velocity we obtain a simple example of a total collision solution of the form $q(t)=f(t)q_0$ where $f(t)>0$ is a scalar function with $f(t)\into 0$ as $t\into T$.  There are many other less obvious examples of total collision but always the limit points of $\hat q$ are CCs.

A normalized CC is just a CC with $I(q) = \normtwo{q}=1$, that is, $q\in S$.  If we think of $\l$ as a Lagrange multiplier, then (\ref{eq_CC}) can be viewed as  the equation for critical points of the Newtonian potential $U(q)$ restricted to the unit sphere $S$.  The problem of infinite spin arises due to the rotational symmetry of the $n$-body problem.  Any configuration $q=(q_1,\ldots,q_n) \ne 0$ determines a circle of symmetrical configurations $R(\t)q=(R(\t)q_1,\ldots,R(\t)q_n)$ where $R(\t)\in SO(2)$ is a $2\times 2$ rotation matrix.   If $q$ is a CC, then so is every
one of its rotates $R(\t)q \in SO(2) q$.   In other words, there are circles of critical points for $U(q)$ on $S$.  Now pass  to the $(2n-4)$-dimensional quotient manifold $S/SO(2)$ 
which is diffeomorphic to the complex projective space $\CP(n-2)$.
Write the quotient  map $S \to \CP(n-2)$ as $q \mapsto [q]$. Write  $[q] \mapsto U([q])$  for the function  on $\CP(n-2)$ induced by $U$. 
 A famous conjecture about CCs is that there are only finitely many of them up to symmetry,  or equivalently, that $U([q])$ has only finitely many critical points.  This is known to be the case for $n=3,4$ and for generic masses when $n=5$.   Whether or not the conjecture holds, we   call $q$ an {\em isolated CC} if $[q]$ is an isolated critical point of $U([q])$ and a nondegenerate CC if it is
 a nondegenerate critical point of $U([q])$.  (Nondegenerate CCs are necessarily  isolated CCs.)    
 
If $q(t)$ is a total collision solution we can form  the curves $\hat q(t)\in S$ and $[\hat q(t)]\in S/SO(2)$.  It follows from Chazy's result that $[\hat q(t)]$ converges to a compact subset of the set of critical points of $U([q])$ as $t \to T$.  If the limit set of $[\hat q(t)]$  contains an isolated critical point $[q_0]$ then it must be that $[\hat q(t)]\into [q_0]$ as $t\into T$.  It follows that $\hat q(t)\in S$ converges to the corresponding circle of normalized critical points in $S$ and it is natural to wonder if $\hat q(t)$ converges to a particular CC or to a nontrivial subset of the circle.
For example it's easy to imagine that by undergoing {\em infinite spin} it might converge to the whole circle.  The main goal of this paper is to show this can't happen.

\begin{theorem}\label{th_nospin}
Suppose $q(t)$ is a solution of the planar $n$-body with $q(t)\into 0$ as $t\into T$ and suppose its reduced and normalized configuration $[\hat q(t)]\in S/SO(2)$ converges to  an isolated CC.  If $\t(t)$ is an angular variable on the corresponding circle of CCs  in S, then $\t(t)$ converges as $t\into T$ and so $\hat q(t)$ converges to a particular CC in the circle.
\end{theorem}
\noindent This theorem was proved by Chazy under the assumption that the limiting CC  is nondegenerate \cite{Chazy} so the problem is really  to handle degenerate cases.
Wintner mentions this as an open problem and it appears as problem 5 in a list of open problems in celestial mechanics \cite{AlbCabSan}.   Here we are solving  the planar case with an isolated  limiting CC.

Theorem~\ref{th_nospin} will follow from another result about solutions in a blown-up and rotation-reduced phase space.  The modern approach to studying total collision uses McGehee's blow-up technique which involves introducing the size coordinate $r$ and a set of coordinates on the unit sphere $S$ as well as suitable velocity variables and a new timescale $\tau$ \cite{McGehee,McGeheeSing}.  In these coordinates $\{r=0\}$ becomes an invariant {\em total collision manifold}.  There are restpoints in the collision manifold associate to the normalized CCs.  Orbits which previously experienced  total collision as $t\nearrow T$ now converge to a compact subset of the set of restpoints in the collision manifold as $\tau\into +\infty$.  In the next section, we will introduce such coordinates together with an angular coordinate $\t$ and reduced coordinates corresponding to the quotient space $S/SO(2)$.  In the end we obtain a blown-up and rotation reduced problem together with an integral formula for the angle $\t$.  Then we will show
\begin{theorem}\label{th_finitearclength}
For the blown-up and rotation reduced problem, total collision solutions which converge to isolated CCs have finite arclength with respect to a natural Riemannian metric.
\end{theorem}
This will imply Theorem~\ref{th_nospin} since the finite arclength implies that the integral giving the change in $\t$ as $\tau\into\infty$ is finite.

Several previous works have claimed to solve this problem by estimating the rotational component of the velocity.  But this does not take into account the ``falling cat'' phenomenon where rotation is produced by changes in shape even though the rotational component of the velocity is zero.  We have included an Appendix which contains a fuller discussion of this issue.  Our approach is based on combining a study of the flow on the center manifold of a degenerate restpoint with an integral formula for the change in angle.  
A recent preprint based on power series expansions for the flow on the center manifolds claims to solve the spin problem for one-dimensional and some two-dimensional center manifolds \cite{XiangYu}.

\section{Reduced Planar $n$-Body Problem and the Spin  Angle}
 Assume that the center of mass is at the origin and the total momentum is zero.    We can parametrize the $(2n-2)$-dimensional center of mass subspace, $\cW$, by a linear map $q=Pz$, $P:\R^{2n-2}\into \cW\subset \R^{2n}$ where $z=(z_1,\ldots,z_{n-1})\in \R^{2n-2}$, $z_i\in\R^2$.
For example, we could use relative positions $z_i= q_i-q_n$ with respect to the $n$-th body or generalized Jacobi variables.  

No matter how these coordinates are  chosen, the mutual distances will be expressible in terms of $z$ and (\ref{eq_U}) determines an analytic function  $U(z)$  on 
$\R^{2n-2}\setminus \Delta$ where $\Delta=\{z: r_{ij}=0\text{ for some }i\ne j\}$.
$-U(z)$ is the Newtonian potential energy function.  The moment of inertia and squared mass norm  $I(q)= \sum_i m_i|q_i|^2$ 
will become
$$I(z)  = \normtwo{z}^2=z^TMz$$
where $M$ is the $(2n-2)\times (2n-2)$ positive-definite symmetric matrix $P^T \diag(m_1,m_1,\ldots,m_n,m_n)P$.   

Let $\z=\dot z$ be the corresponding velocity variables.  Then  since $\dot  q=P\z$ the kinetic energy
$K=\fr12 \sum_i m_i |\dot q_i|^2$ becomes  
$$K(\z) = \fr12\normtwo{\z}^2=\fr12 \z^TM\z$$
Then the translation-reduced problem can be viewed as a Lagrangian system on the tangent bundle  $T(\R^{2n-2}\setminus\Delta)$ with Lagrangian $L = K(\z)+U(z)$.
The Euler-Lagrange equations are
$$
\begin{aligned}
\dot z&=\z\\
\dot \z &= M^{-1}\nabla U(z)
\end{aligned}
$$
and the total energy of  the  system  is
$$K(\z)-U(z) =\fr12\normtwo{\z}^2-U(z) = h.$$

Next we will introduce new coordinates which represent the size, rotation angle and shape of the configuration.  To describe these,  it's convenient to view the   positions and velocities as complex numbers, so we have $z,\z \in\C^{n-1}$.  We will introduce a Hermitian mass metric on  $\C^{n-1}$
$$\cmetrictwo{v,w} = \overline{v}^T M w\qquad v, w\in\C^{n-1}\simeq \R^{2n-2}$$
where $\overline{v}$ is the complex conjugate of $v$.
The real part $\metrictwo{v,w} =\re\cmetrictwo{v,w}$ will be called the mass inner product.  The corresponding norm is just our mass norm  $\normtwo{v}$.
The imaginary part $\im\cmetrictwo{v,w}$, a nondegenerate antisymmetric bilinear form, will also be useful below.

Let $r=\normtwo{z}$ and define the normalized configuration $\hat z  = z/r\in\S^{2n-3}$.  Then we will have
\begin{equation}\label{eq_constraints}
\normtwo{\hat z}=1\qquad \metrictwo{\hat z,\hat{\z}} = 0
\end{equation}
where $\hat\zeta  =\dot {\hat z}$.  To introduce an explicit rotation angle, restrict to the  open subset of $\S^{2n-3}$ where $\hat z_{n-1}\ne 0$ and introduce polar coordinates such that $\hat z_{n-1}=ke^{i\t}$ where $k>0$.   Note that the angle  is chosen so that $\t=0$ represents a state with $z_{n-1}$  in the  positive real axis.  This represents a choice of a  local section for the rotation group action.  Different choices for the angular variable would differ only by a constant shifts on each rotation group orbit and would not affect the question of convergence.

Define new variables $s_i = z_i/z_{n-1} = \hat z_i/\hat z_{n-1}$.   Then  
$$\hat z = ke^{i\t}(s_1,\ldots,s_{n-1},1)= ke^{i\t}(s,1)\qquad s\in\C^{n-2}.$$
Since  $\normtwo{\hat z}=1$  we  have $k = \normtwo{(s,1)}^{-1}$ and our coordinate change is given by $\psi:\R^+\times\S^1\times  \C^{n-2}\into \C^{n-1}$
\beq
z=\psi(r,\t,s) = r \,e^{i\t}\frac{(s,1)}{\normtwo{(s,1)}}.
\Leq{loc section}
Differentiating, we compute that
\beq
\dot z = (\dot r + \dot k r + i\dot \t  r )e^{i\t} k (s, 1) +re^{i\t}k (\dot s, 0)
\Leq{eq: vel}
where we continue to use $k = \normtwo{(s,1)}^{-1}$ from which the chain rule yields 
$$\dot k  = - \frac{G(s,\dot s)}{\normtwo{(s,1)}^3},  \text{ and }  G(s,\dot s) = \metrictwo{(s,1),(\dot s,0)}.$$
Introduce  velocity variables  $\z = \dot z$,   $\r=\dot  r, \w=\dot s$.  Also introduce  
$$\Omega(s, \w) = \im\cmetrictwo{(s,1), (\w, 0)}$$
and  write $\normtwo{\w}$ instead of  $\normtwo{(\w,0)}$.
Then simplifying  we get  the Lagrangian to be
$$L = \fr12\rho^2+\fr12 r^2\dot\t^2 +\frac{r^2\normtwo{\w}^2}{2\normtwo{(s,1)}^2} - \frac{r^2G(s,\w)^2}{2\normtwo{(s,1)}^4} +\frac{r^2\dot\t \,\Omega(s, \w)}{\normtwo{(s,1)}^2} + \frac1r V(s)
$$
where 
$$V(s) =\normtwo{(s,1)}\,U(s,1).$$
The first five terms of the  Lagrangian $L$ are just  the kinetic energy  
 $\frac{1}{2} \normtwo{\z}^2$ rewritten in terms of our new variables.
The formula relating $V$ and $U$   follows from the rotation invariance and homogeneity of the potential:
$$U(z) =U\left(\frac{re^{i\t}(s,1)}{\normtwo{(s,1)}}\right)) =  \frac{\normtwo{(s,1)}}{r}U (s,1)$$

The last step of the reduction uses that the Lagrangian is independent of $\t$ to eliminate $(\t,\dot \t)$.  First compute the angular momentum and solve for $\dot\t$:
\beq
\mu = L_{\dot\t} = r^2\dot\t+ \fr{r^2\Omega(s, \w)}{\normtwo{(s,1)}^2}\qquad \dot\t = \fr{\mu}{r^2} - \frac{\Omega(s, \w)}{\normtwo{(s,1)}^2}.
\Leq{eq: ang mom}
Then the reduced Lagrangian or Routhian is $R_\mu = L -\mu \dot\t$ where $\dot\t$ is to be eliminated in terms of $\mu$.  Carrying out this computation gives
$$R_\mu=\frac12\r^2+ \frac{r^2\normtwo{\w}^2}{2\normtwo{(s,1)}^2} - \frac{r^2G(s,\w)^2}{2\normtwo{(s,1)}^4}   -\frac12\left( \frac{r\Omega(s, \w)}{\normtwo{(s,1)}^2} - \frac{\mu}{r}\right)^2 +\frac1r V(s).$$

A  classical result shows that total collision is possible only when  $\mu=0$ and from now on we will concentrate on this case.  Then the Routhian becomes
$$R(r,\r,s,\w) = \frac12\r^2+\frac{r^2}{2}\normtwo{\w}^2_{FS} +\frac1r V(s)$$
where
\begin{equation}\label{eq_FSnorm}
\begin{aligned}
\normtwo{\w}_{FS}^2 &= \frac{\normtwo{(s,1)}^2\normtwo{(\w,0)}^2 - |\cmetrictwo{(s,1),(\w,0)}|^2}{(\normtwo{(s,1)}^2)^2}\\
& = \frac{\normtwo{\w}^2}{\normtwo{(s,1)}^2} - \frac{G(s,\w)^2+\Omega(s, \w)^2}{(\normtwo{(s,1)}^2)^2}.
\end{aligned}
\end{equation}
We will also use the notation $F(s,\w)=\normtwo{\w}^2_{FS}$.  It is the local coordinate representation of the square of the Fubini-Study norm on the complex projective space $\CP(n-2)$ induced by the  mass norm on $\C^{n-1}$.  There is also a corresponding Fubini-Study metric.  

Since we are using local coordinates we have a Lagrangian system on the tangent bundle of $\R^+\times\C^{n-2}\setminus \Delta$ or $\R^+\times\R^{2n-4}\setminus \Delta$.
Reverting to real coordinates, we can write the Fubini-Study norm as $F(s,\w) = \w^TA(s) \w$ where $A(s)$ is a positive-definite $(2n-4)\times (2n-4)$ matrix.
Then we have
\begin{equation}\label{eq_ELreduced}
\begin{aligned}
\dot r&=\r\\
\dot \r&=rF(s,w)-\frac1{r^2}V(s)\\
\dot s &=\w\\
\dot \w &=\frac{1}{r^3}A^{-1}(s)\nabla V(s) -\frac{2\r \,\w }{r}+ \frac12A^{-1}(s)\nabla F(s,\w)-A^{-1}(s)\dot A(s)\w\end{aligned}
\end{equation}
where $\nabla$ denotes the Euclidean gradient or partial gradient with respect to $s$.
The total energy of  the  system  is
$$\frac12\r^2+\frac{r^2}{2}F(s,\w) -\frac1r V(s) = h.$$
The last equation in (\ref{eq_ELreduced}) follows from the Euler-Lagrange equation $(R_\w)^\cdot =  R_s$.  We have $R_\w = r^2A(s)\w$ so
$$(R_\w)^\cdot  = r^2A(s)\dot \w + r^2 \dot A(s)\w + 2r\rho\, (A(s)\w).$$

The equation for $\dot\w$ can be written more concisely if we make use of  gradients and covariant derivatives with respect to the Fubini-Study metric.  Let 
$\tilde\nabla  = A^{-1}(s)\nabla $ denote the gradient or partial gradient with respect to the Fubini-Study metric and let $\tilde D_t$ denote the covariant time derivative of a vectorfield along a curve $s(t)$.  Then we have
$$\begin{aligned}
\tilde D_t \w &= \dot \w - \frac12A^{-1}(s)\nabla F(s,\w)+A^{-1}(s)\dot A(s)\w\\
&=\dot \w - \frac12\tilde\nabla F(s,\w)+A^{-1}(s)(DA(s)(\w))\w
\end{aligned}
$$
and the last Euler-Lagrange equation simplifies to
\begin{equation}\label{eq_lastEL}
\tilde D_t \w = \frac{1}{r^3}\tilde\nabla V(s)-\frac{2\r \,\w }{r}.
\end{equation}

Even though we eliminated the angle $\t$, we want to study its behavior for solutions approaching the total collision singularity.  Although the angular momentum is zero, $\t(t)$ need not be constant.  In fact we have
\beq
\dot\t =- \frac{\Omega(s,\w)}{\normtwo{(s,1)}^2}.
\Leq{horiz lift}
It's conceivable that as the shape $s(t)$ changes, the integral of $\dot\t(t)$ could diverge producing ``infinite spin' (see the Appendix for a fuller explanation).   We will show that this does not happen.

\section{Collision Manifold and Center Manifold}
To study orbits converging to total collision we will use McGehee's blowup method.  Namely, introduce rescaled velocity variables $v=\sqrt{r}\r, w=r^\fr32 \w$ and a new time variable $\tau$ such that ${}'=r^\fr32 {}^\cdot$.   The blown-up differential equations are
\begin{equation}\label{eq_blowup}
\begin{aligned}
r'&=vr \\
v'&= \fr12 v^2+F(s,w)-V(s)\\
s'&=w\\
\tilde D_\tau  w &= \tilde\nabla V(s)-\frac12 vw \end{aligned}
\end{equation}
where
$$\tilde D_\tau w = w' - \frac12\tilde \nabla F(s,w)+A^{-1}(s)(DA(s)(w))w.$$
The energy equation is
$$\fr12v^2+ \fr12 F(s,w)-V(s)=rh.$$

The collision manifold $\{r=0\}$ is invariant.  The equilibrium points are of the form $P = (r,v,s,w) = (0,v_0,s_0,0)$ where $\tilde\nabla V(s_0)=0$.  The last equation, which is equivalent to $\nabla V(s)=0$, characterizes the reduced central configurations.  If $q(t)$ is a total collision solution with $t\nearrow T$ then the corresponding solution
$\g(\tau) =(r(\tau),v(\tau),s(\tau),w(\tau))$ converges to a compact subset of the set of equilibrium points in some level set $v=v_0<0$.  The energy equation shows that $v_0=-\sqrt{2V(s_0)}$.  Since $r'=vr$ and $v(\tau)\into v_0<0$, it follows that $r(\tau)$ converges to 0 exponentially.

Suppose $P$ is an isolated equilibrium point with $v_0<0$ which is one of the limit points of such a total collision solution, $\g(\tau)$.   
Since $P$ is isolated we have $\g(\tau)\into P$ as $\tau\into\infty$.   We want  to show that  the  Fubini-Study arclength $L(s)$ is finite, where 
\begin{equation}\label{eq_arclength}
L(s) = \int \normtwo{s'(\tau)}_{FS}\,d\tau= \int \normtwo{w(\tau)}_{FS}\,d\tau.
\end{equation}

The linearized differential equations  at $P$ have matrix
$$ 
\m{\delta r'\\ \delta v'\\ \delta s'\\ \delta w'}=\m{v_0&0&0&0 \\ 0&v_0& 0&0\\0&0&0& I \\ 
0&0&D\tilde\nabla V(s_0)&-\fr12 v_0 I}\m{\delta r\\ \delta v\\ \delta s\\ \delta w}.
$$
Since $\nabla V(s)=0$, it follows that  that $D\tilde\nabla V(s_0)= D(A^{-1}(s)\nabla V(s)) = A^{-1}(s)D\nabla V(s)$.

The tangent space to the  energy manifold is given by $v_0\delta v=h\delta r$ so the upper left  $2\times 2$  block gives rise to a single eigenvalues $v_0<0$.  This corresponds to the exponential convergence of $r(\tau)$ to 0.  Let $B$ be the lower right $(4n-8)\times (4n-8)$ block, representing the linearized differential equations within the collision manifold.  If $\delta s$ satisfies $D\tilde\nabla V(\s_0)\delta s = c\, \delta s$ then $(\delta s, \delta w) =(\delta  s, \l_\pm\delta s)$ is an eigenvector of $B$ with  eigenvalue
$$\l_\pm =\fr{-v_0\pm \sqrt{v_0^2+16 c}}{4}.$$
Since  $v_0<0$, it follows that any nonreal eigenvalues are unstable.  Also $\re \l_+>0$ and we have $\l_-=0$ if and only if $c=0$.

If the matrix $D\nabla V(s_0)$ is nonsingular,  then the corresponding restpoint $P$ is  hyperbolic and therefore isolated.  Any solutions approaching $P$ as $\tau\into\infty$ are in the stable manifold  and converge exponentially fast.  From this, it follows that the integrand of the arclength integral (\ref{eq_arclength}) converges to 0 exponentially and therefore $L(s)<\infty$.  Since the angular momentum is zero, the spin angle $\t(\tau)$ satisfies 
\begin{equation}\label{eq_thetaprime}
\t'= - \frac{\Omega(s, w)}{\normtwo{(s,1)}^2}.
\end{equation}
This admits an estimate of the form $|\t'|\le K\normtwo{w}_{FS}$ so $\t(\tau)$ also converges  to a limit as $\tau\into\infty$.  This proves Theorems~\ref{th_nospin}  and \ref{th_finitearclength} when the omega limit set contains a nondegenerate CC.

Suppose now that there is a degenerate CC with corresponding restpoint $P$ in the limit set.  By choice of local coordinates, it's possible to assume that $s_0 = 0\in \R^{2n-4}$.  This is equivalent to choosing the unreduced coordinates $z \in \C^n$ so that the coordinates of the chosen central configuration are of the form $z_0 = (0,\ldots, 0,z_n)$, which is easily arranged by a change of complex basis.  We will need to use the center manifold and center-stable foliation for which we refer to \cite{Bressan, Fenichel, Wiggins}.  Modify the differential equations by introducing a cutoff function so that the new differential equations are linear outside of some neighborhood of the origin.  The modified equations have invariant center-stable, center-unstable and center manifolds, tangent to the corresponding subspaces.   These are unique but depend on the choice of cutoff function. However, the  solutions  of interest  will be  contained in the  center-stable manifold no  matter how the cutoff is done.   Furthermore, the center-stable manifold $W^{cs}$ is foliated over the center manifold $W^c$ so that solutions in $W^{cs}$ approach the corresponding solution in $W^c$ exponentially.  

The rest  of the proof will focus on estimating the arclength of solutions $\g(\tau)$ which start close to $P$ and lie in the center manifold $W^c(P)$.  Later  it will be  straightforward to handle solutions in the center-stable manifold.  One technical problem is that center manifolds may only have finite smoothness even though the original differential equations are real analytic.  

\section{Flow on the Center Manifold}
From the discussion of eigenvalues above, the center subspace is
$$\E^c = \{(\delta r,\delta v,\delta s,\delta w)= (0,0,\delta s,0): D\nabla V(s_0)\delta s = 0\}.$$
Let $\cK\subset \R^{2n-2}$ be the kernel of $D\nabla V(0)$.  Suppose it has dimension $k>0$ and suppose the $s$-coordinates are chosen so that 
$\cK=\R^k\times \{0\}$.  Then write $s=(x,y) \in\R^k\times \R^l$, $l=2n-4-k$ and similarly split $w$ as $w=(\x,\y)$.  The center  manifold will take the form of a graph
\begin{equation}\label{eq_centermanifold}
y= f(x)\qquad \x =\phi(x)\qquad \y=\psi(x)
\end{equation}
where $f(0)=0, Df(0)=0$ and similarly for $\phi, \psi$.   Since the collision manifold and energy manifolds  are invariant  we can assume that our center manifold also satisfies
$$r=0\qquad v(x) = -\sqrt{2V(s(x))-\normtwo{w(x)}^2}.$$

Since the center manifold is invariant, the parametrization (\ref{eq_centermanifold}) can be used to pull-back the restriction of the blown-up Euler-Lagrange equations (\ref{eq_blowup}) to get a first-order  differential equation on some neighborhood $\cU$ of the origin in $\R^k$.   If $x(\tau)$ is a solution of this pull-back equation then 
$$s(\tau) = (x(\tau),f(x(\tau)))\qquad w(\tau) = (\phi(x(\tau)),\psi(x(\tau)))$$
is  a solution  of  (\ref{eq_blowup}).  It follows  that
$$x'=\xi  =\phi(x)\qquad  y'=\eta  =\psi(x) = Df(x)\phi(x).$$
These are to hold along every  solution curve  $x(\tau)$ but since  there is a solution curve through every point $x\in \cU$, they can be regarded as functional equations satisfied by $f,\phi,\psi$.  In particular, our pull-back differential equation is just $x'=\phi(x)$.  

If we introduce  the notation  $g(x)=(x,f(x))$ the functional equation for $\psi$ shows  that the  parametrization of the center manifold (\ref{eq_centermanifold})  can be written 
$$s = g(x)\qquad w = (\phi(x),\psi(x)) = Dg(x)\phi(x).$$
Let $x(\tau)$ be any solution of the pull-back differential $x'=\phi(x)$  and let $w(\tau) = Dg(x)\phi(x)|_{x=x(\tau)}$.  Then the chain rule gives
$$w'(\tau) = Dg(x)D\phi(x)\phi(x) + D(Dg(x))(\phi(x))\phi(x)$$
where $x=x(\tau)$ and the covariant derivative in (\ref{eq_blowup}) takes the form
$$\tilde D_\tau w(\tau) = w'(\tau) - \frac12\tilde \nabla F(g(x),w(x))+A^{-1}(g(x))(DA(g(x))(w(x)))w(x)$$
where $w(x) = Dg(x)\phi(x)$ and $x=x(\tau)$.  Substituting all this into (\ref{eq_blowup}) gives an estimate which will be needed later.
\begin{lemma}\label{lemma_nablaVestimate}
In a sufficiently small neighborhood $\cU$ of the origin in $\R^k$, there are positive constants $k_1,k_2$ such that
\begin{equation}\label{eq_gradVestimate}
k_1 |\phi(x)|\le  |{\tilde\nabla V(g(x))}|\le k_2 |\phi(x)|.
\end{equation}
\end{lemma}
\begin{proof}
Let $x(\tau)$ be any solution of  $x'=\phi(x)$ and $w(\tau) = Dg(x)\phi(x)$.  Then (\ref{eq_blowup}) gives 
\begin{equation}\label{eq_Dtauw}
\tilde D_\tau w =\tilde\nabla V(g(x)) - \fr12 v(x)Dg(x)\phi(x)
\end{equation}
where $v(x) = -\sqrt{2V(g(x))-\normtwo{Dg(x)\phi(x))}^2}$ and $x=x(\tau)$.  The equations for $w'(\tau)$  and $D_\tau w(\tau)$ above give a formula for the left-hand side of  (\ref{eq_Dtauw}) in terms of $g(x(\tau))$  and $\phi(x(\tau))$.  Since  there is a solution $x(\tau)$ through any given $x\in \cU$ we can obtain a functional equation in $x$.  
Note that most of the terms on the left-hand  side depend quadratically on $\phi(x)$  with  one exceptional  term depending on $D\phi(x)\phi(x)$.  
Recall that $\phi(0)=0$ and $D\phi(0) = 0$ and observe that the matrix norm $|Dg(x)|$ is bounded.  It follows that given any constant $c>0$ we have an estimate of the form  
$|\tilde D_\tau w| \le c|\phi(x)|$ in every sufficiently small neighborhood of the origin. 

Since $v(x)\simeq -\sqrt{2V(0)} <0$, the last term in (\ref{eq_Dtauw}) will satisfy $c_1|\phi(x)| \le |\fr12 v(x)Dg(x)\phi(x)| \le c_2|\phi(x)|$ for some positive constants  $c_1,c_2$.  Then  we  have
$$(c_1-c)|\phi(x)| \le |\tilde \nabla  V(g(x))| \le (c_2+c)|\phi(x)|.$$
Working in a sufficiently small neighborhood  we can get $c<c_1$ to complete the proof.
\end{proof}

The functional  equation (\ref{eq_Dtauw})  involves the Fubini-Study gradient $\tilde\nabla V(s)$ in $\R^{2n-2}$, evaluated at $s=g(x)$.  The next step is to reformulate this equation in terms of a k-dimensional gradient of the function $W(x) = V(g(x))$.  Let $X$ denote the projection of the center manifold to configuration space.  The function $s=g(x)= (x,f(x))$ is an immersion which parametrizes $X$  so we can use $g$ to pull-back the Fubini-Study metric to a Riemannian metric on $X$.  Denote the squared norm of this metric by $\normtwo{\xi}^2_X = \hat F(x,\xi)$  where $x,\xi\in\R^{k}$.  Then
$$\normtwo{\xi}^2_X = \hat F(x,\xi) = F(g(s),Dg(x)\xi) = \normtwo{Dg(x)\xi}_{FS}^2$$
where $F(s,w)$  is the squared Fubini-Study norm  (\ref{eq_FSnorm}).
Explicitly, we have $\hat F(x,\xi) = \xi^T B(x) \xi$ where $B(x)$ is the $k\times k$ matrix
$$\hat A(x) = Dg(x)^T A(g(x)) Dg(x)\qquad  Dg(x) =  \m{I \\Df(x)}$$
and $A(s)$ is the matrix for the Fubini-Study metric.

The pull-back metric will determine a covariant derivative which we will denote by $\hat D$. 
Now let $x(\tau)$ be curve in $X$ and $\xi(\tau)$ a  vectorfield along  $x(\tau)$.  There will be a corresponding curve and vectorfield in  $\R^{2n-4}$ given  by
$s(\tau) = g(x(\tau))$, $w(\tau) =  Dg(x(\tau))\xi(\tau)$.  Since $g$ is an isometric immersion, it can be shown that the covariant derivatives $\hat D_\tau \xi$ in $X$ and $\tilde D_\tau w$ in $\R^{2n-4}$ are related by
$$\pi(\tau)\tilde D_\tau w(\tau) = Dg(x(\tau)) \hat D_\tau  \xi(\tau)$$
where $\pi(\tau)$ is the orthogonal projection from $T_{s(\tau)}\R^{2n-2}\into \im(Dg(x(\tau))$ with respect to the Fubini-Study metric.  We can also pull-back the potential function  $V(s)$ to the function $W(x) = V(g(x))$.  If $\hat\nabla$ denotes the gradient with respect to the pull-back metric  on $X$, we have
$$\pi(\tau)\tilde\nabla V(s(\tau)) = Dg(x(\tau))\hat\nabla W(x(\tau)).$$

Let $x(\tau)$ be any solution of  $x'=\phi(x)$ and let $\xi(\tau) = x'(\tau) = \phi(x(\tau))$.  Then $s(\tau) = g(x(\tau))$, $w(\tau) =Dg(x(\tau))\xi(\tau)$ satisfies (\ref{eq_Dtauw}).
Applying the orthogonal projection $\pi(\tau)$ to both sides of (\ref{eq_Dtauw}) gives
$$Dg(x)\hat D_\tau  \xi = Dg(x )\hat\nabla W(x)-\frac12 v Dg(x)\xi.$$
Since $g$ is an immersion, we conclude that $x(\tau),\xi(\tau)$ satisfies
\begin{equation}\label{eq_Dtauxi}
\hat D_\tau \xi = \hat\nabla W(x) - \fr12 v\phi(x).
\end{equation}
Since  $\xi'(\tau)=D\phi(x(\tau))\phi(x(\tau))$, the explicit formula for the covariant derivative $\hat D$ is
$$\hat D_\tau \x = D\phi(x)\phi(x) - \frac12\hat \nabla \hat F(x,\xi)+{\hat A}^{-1}(g(x))(D{\hat A}(x)(\xi)\xi$$
where $\xi =  \phi(x)$ and $x=x(\tau)$.  As above, we can replace $x(\tau)$ by $x$ and view this as a functional equation.  Then an argument analogous to the proof of Lemma~\ref{lemma_nablaVestimate} gives
\begin{lemma}\label{lemma_nablaWestimate}
In a sufficiently small neighborhood $\cU$ of the origin in $\R^k$, there are positive constants $l_1,l_2$ such that
\begin{equation}\label{eq_gradWestimate}
l_1 |\phi(x)|\le  |{\hat\nabla W(x)}|\le l_2 |\phi(x)|.
\end{equation}
\end{lemma}

We can also use (\ref{eq_Dtauxi}) to see that the differential equation on the center manifold is approximately a gradient.  This will be the key to ruling out infinite spin.
\begin{lemma}
The differential equation on the center manifold is $x' = \phi(x)$ where  $\phi(x) = -k\hat\nabla  W(x) + \g(x)$ where $k =-2/v(0)>0$ and $\g(x) = o(|\hat \nabla W(x)|)$.
\end{lemma}
\begin{proof}
The left-hand side of (\ref{eq_Dtauxi}) is $o(|\phi(x)|)$ so we can rearrange the equation to get
$$\phi(x) = \frac{2}{v(x)}\hat\nabla W(x) +o(|\phi(x)|).$$
Also  $v(x) = v(0) +O(|\hat\nabla W(x))+O(|\phi(x)|)$ and $|\phi(x)|\le l_2 |\hat\nabla W(x)|$  so
$$\begin{aligned}
\phi(x) &= -k\hat\nabla W(x) +O(|\hat\nabla W(x)|^2)+O(|\phi(x)||\hat\nabla W(x)|)+o(|\phi(x)|)\\
& = -k\hat\nabla W(x) +o(|\hat \nabla W(x)|).
\end{aligned}
$$
\end{proof}

The final ingredient is a \L{}ojasiewicz gradient inequality for $W(x)$.
\begin{lemma}
In a sufficiently small neighborhood of the origin, the restricted potential $W(x)$ satisfies  
\begin{equation}\label{eq_Loj}
|\hat\nabla W(x)|^2\ge |W(x)-W(0)|^\a
\end{equation} where $1<\a<2$.
\end{lemma}
\begin{proof}
Since $V(s)$  is real analytic, it satisfies a \L{}ojasiewicz gradient inequality $|\nabla V(s)|^2\ge C |V(s)-V(0)|^\a$ in some neighborhood of  $s=0$ where $0<\alpha<2$  and $C>0$. Replacing $\nabla$ by $\tilde\nabla$ would just modify the constant $C$.   It is no loss of generality to assume $1<\a<2$ since when $|V(s)-V(0)|\le 1$, increasing the exponent makes the inequality weaker.  

Since  center manifolds are not necessarily analytic, we can't immediately get such an inequality for $W(x)= V(g(x))$.  However, when $s=g(x)$ we have $V(s)-V(0)  = W(x)-W(0)$.  It follows from Lemmas~\ref{lemma_nablaVestimate} and \ref{lemma_nablaWestimate}  that $|\hat \nabla W(x)| \ge K |\tilde \nabla V(g(x))|$ with  $K=l_1/k_2$.  This gives $|\hat \nabla W(x)|\ge KC|W(x)-W(0)|^\a$ and by a further increase of $\a$ we can arrange that  $KC=1$.
\end{proof}

\L{}ojasiewicz used his inequality to show that for analytic gradient differential equations, solutions converging to a critical point have finite arclength \cite{Loj, ColdingMinicozzi}.  The same property holds for our equation on the  center manifold.
\begin{lemma}
Suppose $x	(\tau)$ is a solution of a differential equation of the form $x'= -k \hat\nabla  W(x) + \g(x)$ where $k>0$ and $\g(x) = o(|\hat \nabla W(x)|)$ and suppose that 
$W(x)$ satisfies an inequality of the form  (\ref{eq_Loj}).  Suppose $x(\tau)$ is a solution with $x(\tau)\into 0$ as $\tau\into\infty$.  Then the arclength of the curve $x(\tau)$ is finite.  Here the gradient and arclength are taken with  respect to some smooth Riemannian metric.
\end{lemma}
\begin{proof}  
Without loss of generality we can assume that $W(0)=0$.   Given any constant $c>0$ we can work in a neighborhood $\cU$ of $x=0$ where (\ref{eq_Loj}) holds and such that 
$|\g(x)| \le c |\hat\nabla  W(x)|$.  We will assume that $c<k$.  
Since $x(\tau)\into 0$ we can assume that $x(\tau)\in\cU$ for all $\tau\ge 0$.

To estimate the arclength, note that  since
$|x'|\le k|\hat\nabla W(x)|+|\g(x)|\le (k+c)|\hat\nabla W(x)|$, it suffices to estimate the integral of $|\hat\nabla W(x(\tau))|$.  First consider $W(x(\tau))$.  We have
$$W' = -k|\hat\nabla W|^2 + \metric{\g,\hat\nabla W}\le -(k-c)|\hat\nabla W|^2\le -k_2 W^\a$$
where $k_2=k-c>0$.  The first inequality shows that $W'(\tau)\le 0$ and  since  $W(x(\tau))\into  W(0)=0$, we must have  $W(x(\tau))\ge 0$ for all $\tau\ge 0$.
This allows us  to drop the absolute value in (\ref{eq_Loj}) to obtain the second inequality.  

Integrating this over $[0,\tau]$ gives
\begin{equation}\label{eq_Woftau}
W(\tau)\le  \frac{W_0}{(1+\l \tau)^\fr{1}{\a-1}}
\end{equation}
where $W_0=W(x(0))$ and  $\l = k_2(\a-1)W_0^{\a-1}$.
Since $1<\a<2$, we have $\frac{1}{\a-1} = 1+2\e$ for some $\e>0$.  

Since $W'\le -k_2|\hat\nabla W(x)|^2$ we can use the Cauchy-Schwartz inequality in a tricky way \cite{ColdingMinicozzi} to find
$$
\begin{aligned}
\int_0^T |\hat\nabla W(x(\tau)| &\le \int_0^T \sqrt{-W'/k_2}=\int_0^T \sqrt{-W'/k_2}\,\tau^\fr{{1+\e}}{2}\,\tau^{-\fr{1+\e}{2}}\\
&\le  \left(\fr1{k_2}\int_0^T -W'\,\tau^{1+\e}\right)^\fr12   \left(\int_0^T\tau^{-(1+\e)}\right)^\fr12.
\end{aligned}
$$
It suffices to show that the first integral is bounded.  Integrating by parts gives
$$\int_0^T -W'\,\tau^{1+\e} = W(T)\tau^{1+\e}|_T^0 + (1+\e)\int_1^T W(\tau)\tau^\e.$$
Since $W(\tau)$ satisfies (\ref{eq_Woftau}) with exponent $1+2\e$, both terms are bounded
\end{proof}

Putting these lemmas together we get the following description of the flow in the local center manifold.
\begin{theorem}\label{th_centerflow}
Suppose $P$ is a degenerate restpoint with $v_0<0$. Let  $\cU$ be a sufficiently small neighborhood of $P$ and let $W^c_\cU$ be the local center manifold.  
If $\g(\tau)$ is a solution in $W^c_\cU$ which converges to $P$ as $\tau\into\infty$, then $\g$ has finite arclength.
\end{theorem}

\section{Completion of the proofs}
To complete the proof of Theorems~\ref{th_finitearclength}, consider any total collision orbit whose reduced and normalized orbit converges to an isolated CC and let $\g(\tau)$ be the corresponding solution of the  blown-up,  reduced equations (\ref{eq_blowup}).  Then $\g(\tau)$ converges to a restpoint $P=(0,v_0,s_0,0)$ on the collision manifold with $v_0<0$ and with $s_0$ representing the isolated CC.  We have already  considered  the case  where  $s_0$ is a  nondegenerate critical  point so suppose $s_0$  is degenerate.  Choose a neighborhood of $P$ where Theorem~\ref{th_centerflow} is valid.  We may assume that $\g(\tau)\in\cU$ for all $\tau\ge 0$ and it follows that $\g(\tau)$ is in the local center-stable manifold  $W^{cs}_\cU$ for $\tau\ge 0$.  The center-stable fibration implies that there will be an associated orbit $\g_c(\tau)$ in the center manifold on the same fiber.   Since the fibers of the center-stable fibration are exponentially contracting we also have $\g_c(\tau)\into P$ as $\tau\into\infty$ and Theorem~\ref{th_centerflow} implies that $\g_c(\tau)$ has finite arclength.  To see that $\g(\tau)$ also has finite arclength note that since the distance between $\g(\tau)$ and $\g_c(\tau)$ converges to zero exponentially and since both curves are solutions of a smooth first-order differential equations, the distance between derivatives  $\g'(\tau)$ and $\g'_c(\tau)$ also converges to zero exponentially.  It follows that difference of their arclengths is bounded.  

This completes the proof of Theorem~\ref{th_finitearclength}.  As explained above, the finite arclength implies that the integral defining the spin angle $\t(\tau)$ converges which gives Theorem~\ref{th_nospin}.

\section{Appendix} 

We point out  an  error in
previous claims of  proofs that  total
collision solutions to the N-body problem cannot have infinite spin.  See \cite{Saari-H} and \cite{Elbialy}.
Those proofs rely on a configuration-dependent projection of 
  velocity vectors onto their  ``rotational part''.  See definition \ref{def: horiz} below. The authors
 correctly prove that  the velocities of their colliding solutions have zero  rotational part. (In the case of partial
 collisions they prove that these ``rotational parts'' tend exponentially fast to zero.)
The error made is in   an  implication drawn from this vanishing.  The papers  claim that if the velocities have zero rotational part
then the net rotation suffered during the motion must be finite.  We will show through explicit examples
that the   rotational part of  velocity can be identically  zero and that nevertheless 
the n-body configuration can suffer infinite rotation, i.e., infinite  spin.

A sensible reader might complain:   ``If the rotational part of a motion   is zero how can the  object  spin at all, let alone have infinite spin?''   If this object  were a rigid body then indeed this complaint is legitimate, such a body  cannot spin at all when the rotational part of its motion is zero.   
But our objects are not rigid,   rather they are  ``constellations'' -- configurations of n moving points   in the plane. 
Making sense of net rotation or spin  for non-rigid bodies  is a rather subtle matter.
At the heart of the matter are ideas from gauge theory as formalized by  the theory of principal bundles with connections as we now describe
by reverting to the falling cat.  (See \cite{mont_monthly}, \cite{mont_isohol} and \cite{mont_tour}  for more details on this perspective on
rotation and re-orientation of non-rigid bodies and its relation to the n-body problem.) 

To set the stage it is important to know that   the ``rotational part'' of our velocity is zero if and only if the total angular momentum
of the motion is zero.  See proposition \ref{lemma: ang mom} below.  
  A falling cat, dropped  with zero angular momentum, will perform a net rotation  and right itself.
Despite the fact that the rotational part of the velocity of the cat's configuration is zero, nevertheless it still spins
enough to right itself.    
The cat cannot change its angular momentum (as  viewed from its center of mass).  Angular momentum is conserved.
 It rights itself by changing its   shape
and  taking advantage of the fact that 
  ``angular momentum equals zero'' defines a connection with non-zero curvature.  

We  think of our moving n-body system as a kind of  falling cat.  Setting the 
total angular momentum to zero defines a connection on the principal circle bundle
$S = \S^{2n-3}  \to \CP(n-2)$.   (The equation of parallel transport for this connection is equation  
(\ref{eq_thetaprime}) in the body of the paper.)  
{\it The key to generating infinite spin    is  that the   curvature 
of this connection is not zero. }   As a consequence, we  can draw curves on $\CP(n-2)$ which limit
in infinite time to a fixed point $s_0$ but whose horizontal lifts have no limit, but rather
contain the whole fiber over $s_0$ in their closure. 

Write $\R^{2n}$ for the planar n-body  configuration space.  The group G of translations, scalings and rotations
acts on $\R^{2n}$.   These three generating subgroups of $G$ define   three linear subspaces of $\R^{2n}$ which we call
(translation), (scaling),  and (rotation).  The last two subspaces depend on the configuration $q \in \R^{2n}$.
 See below for details on the subspaces.  
  These three subspaces  do not exhaust $\R^{2n}$ 
if  $n > 2$.   The orthogonal complement of their direct sum represents the rest of configuration
space and vectors in this complement represent ``pure shape'' deformations so we call that the pure shape
subspace. 
Altogether then we get
  the ``Saari decomposition''  of velocity space: 
   \beq
   \R^{2n} =  \text{translation} \oplus \text{scaling} \oplus \text{rotation} \oplus \text{pure shape}.
   \Leq{Saari}
It is essential that we compute the
orthogonal complement defining (pure shape) relative to the mass metric.
 $$\langle v ,w \rangle = \Sigma m_a v_a \cdot w_a$$
 on $\R^{2n}$.   
 
 {\sc Remark.} Saari formalized this  decomposition 
  in  \cite{Saari_decomp} hence his name became attached to it. 
 The first three subspaces are   mutually orthogonal provided the center of mass of
 the configuration is zero. 
 
 \begin{definition}  The `rotational part' of a vector $v \in \R^{2n}$
 is its orthogonal projecton onto the rotation subspace.
 The vector is called `horizontal' if it lies in the pure shape subspace.
 A curve $q(t) \in \R^{2n}$ is called horizontal if its derivative $\dot q(t)$ is everywhere horizontal.
 \label{def: horiz}
 \end{definition}
 
 When we fixed the center of mass at the origin we got rid of the translational degrees of freedom
 and hence the translation subspace of $\R^{2n}$.   Recall that we identify the center-of-mass zero
 subspace of $\R^{2n}$ with $\C^{n-1}$ in which case the Saari decomposition becomes
   \beq
   \C^{n-1}  =  \text{scaling} \oplus \text{rotation} \oplus \text{pure shape}.
   \Leq{Saari B}
   and the mass metric is the real part of the Hermitian mass metric.  
   The scaling and rotation  groups act on this $\C^{n-1}$ by scalar multiplication,
   with scaling acting by   $z \mapsto \lambda z$,  $\lambda > 0$ real,  while 
  the rotation group acts by $z \mapsto e^{i \t} z$, with $\t$ the rotation angle.
  Differentiating these actions  at the identity, or what is the same, take the tangent space
  to their orbits at $z$  yields the scaling and rotation subspace.
  So the  
 scaling space consists of the real span of $z$ while the rotation subspace
 consists of the real span of $iz$.

 \begin{proposition} \label{lemma: ang mom} The rotational part of a velocity vector $\z$ attached at $z$ is zero
 if and only if the total angular momentum $J(z,\z)$ is zero. Moreover,
$$J(z, \z) = \im\cmetrictwo{z,\z}$$
 \end{proposition}
   
{\sc Proof.}   In order to make the expressions conform
more closely to standard physics usage we will use $v$ for `velocity' instead of
$\z$.  The rotational part of $v = \z$ is zero if and only if it is orthogonal to the rotational subspace which   is the real span of   $i z$. 
So  $v$ has zero rotational part if $\re\cmetrictwo{iz,v} = 0$.  But $\re\cmetrictwo{iz,v} = - \im\cmetrictwo{z, v}$.
  Expanding out  
\begin{eqnarray*} - \im\cmetrictwo{z, v} & = & - \Sigma m_a \im (\bar z_a v_a) ) \\
& = &  - \Sigma m_a z_a \wedge v_a 
\end{eqnarray*}
where in the middle line  the wedge denotes the planar version of the
cross product:  $(x, y) \wedge (a, b) =    \im (x- iy) (a+ib) = xb-ya$
for $(x,y), (a,b) \in \R^2$.    This final expression will be recognized as the
standard expression for planar angular momentum:
$$  \Sigma m_a z_a \wedge v_a =  J(z,v).$$
QED

{\sc remark}  The orthogonal projection of $z$ onto the rotational subspace at $z$  is $i \omega z$
where the scalar $\omega = \frac{1}{I(z)} J(z, iz)$.    

We rewrite the angular momentum in terms of our variables $r, \theta, s$ introduced in
the equation  (\ref{loc section})):   
$z = rk e^{i \t} (s,1)$ with $k = \frac{1}{\normtwo{(s,1)}}$
 and the corresponding velocities  
$$\z = \dot z = (\frac{\dot r}{r} + \frac{\dot k}{k} + i\dot \t   )z  +re^{i\t}k (\dot s, 0).$$
It follows that $\cmetrictwo{z,\z} = ((\frac{\dot r}{r} + \frac{\dot k}{k} + i\dot \t   )r^2 + r^2 k^2 \cmetrictwo{(s,1), (\dot s, 0)}$ and 
from $J(z, \z) = \im \cmetrictwo{z, \z}$ we get
$$J(z,\z) = r^2 \dot \theta + r^2 k^2 \Omega(s, \dot s) ,  \text{ where } \Omega(s, \dot s) : = \im\cmetrictwo{(s,1), (\dot s, 0)}.$$

It follows that
\beq
 J(z, \z) = 0  \iff   \dot \theta = - \frac{\Omega(s, \dot s)}{ \normtwo{(s, 1)}^2} 
 \Leq{eq: ang mom 2} 
in agreement with equation (\ref{eq: ang mom}) where the angular momentum was denoted $\mu$. 

It will help in what follows to simplify this last  expression,  by  applying a linear transformation
to the   original  complex linear coordinates $z_i$  which renders our 
by applying a Hermitian mass metric in standard form. 
In these new variables, which we continue to call $z_i$, we have that   $\cmetrictwo{z,  z} = \Sigma \bar z_i  z_i$.
We similarly redefine the $s_i$ as the quotients of the new $z_i$ by the new $z_{n-1}$. 
Then with the $s_i$ similarly redefined: 
$$\frac{\Omega(s, \dot s)}{ \normtwo{(s, 1)}^2} =  \frac{ \im \Sigma \bar s_i \dot s_i}{ \Sigma_{i=1}^{n-2}  |s_i|^2 + 1 } $$
 
  Here then, is the main result of this appendix.
     \begin{proposition}  Given any $s_0 \in \CP(n-2)$ there are analytic curves
   $c(t)$ converging to $s_0$ in infinite time and whose horizontal lifts $z(t)$
   to $\C^{n-1} \setminus 0$   have infinite spin.
   \end{proposition} 
   \noindent Horizontal curves all have zero rotational part. The
   proposition shows that    having zero rotational part does not guarantee finite spin.

   {\sc Proof.}  We may take $s_0$ to have affine coordinates $s_i = 0$
   and so represented by the point with  homogeneous coordinates $[0, \ldots , 0, 1]$. 
   Any   curve whose projection to $\CP (n-2)$ converges to $s_0$
   is represented in affine coordinates by a curve $s(t) = (s_1 (t), \ldots , s_{n-2} (t))$
   with the $s_i (t) \to 0$ as  $t \to \infty$.  
   Any curve in $\C^{n-1} \setminus \{0\}$ which projects onto this curve has
   the form  $(r(t), \theta (t),  s(t))$ in our coordinates.  The curve is horizontal if and only if
   $\theta$ satisfies the differential equation (\ref{eq: ang mom 2}).  
   We integrate
   this  zero angular momentum equation to find that  
   $$\theta(t) - \theta (0) = - \int_{s([0, t])}  \frac{ \Omega(s, ds)}{\|s \|^2 +1}.$$
     where we set  $\Sigma_{i=1}^{n-2}  |s_i|^2 = \|s \|^2$.
     
In order to achieve the desired  curves $c(t)$ of the lemma  we only need to vary the   first
   affine coordinate $s_1$, which is to say, we take  curves of the form
   $s(t) = (s_1 (t), 0, \ldots, 0)$.  For such a curve the integrand occuring in the line integral is  
   $\frac{\im (\bar s_1 ds_1 ) }{|s_1|^2 + 1}.$  For the purposes of the proof, we set 
   $$s_1 = r e^{i \psi}$$
   in which case $\im \bar s_1 ds_1 = r^2 d \psi$ and $|s_1|^2 +1 = r^2 + 1$
   so that our integrand is
   $$ \frac{ \Omega(s, ds)}{|s|^2 + 1} = \frac{ r^2 d \psi }{r^2 + 1}.$$
   Now suppose our curve $s_1 (t)$  lies in the unit disc $r < 1$ so that
   $1/(r^2 + 1) > 1/2$  and  that our curve also  spirals counterclockwise into the origin of the $s_1$ plane so 
   that $\psi (t)$ increases
   monotonically with $t$.  We then
   get
   $$\theta(t) <  \theta (0) - \frac{1}{2}  \int_0 ^t  r(t)^2 \dot \psi  dt.$$

To produce curves $c(t)$ with infinite spin we are left with an easy task:   insure
that $r(t) \to 0$ as $t \to \infty$  while the integral $\int_0 ^t  r(t)^2 \dot \psi  dt$ diverges to $+ \infty$.  
 As one class
of examples, take $r(t) = \frac{1}{\sqrt{t}}$ and $\psi = ct$ for any positive  constant $c$.
Then $r^2 \dot \psi = \frac{c}{t}$ and the angle $\theta(t)$
diverges logarithmically to minus  infinity.   
QED

\end{document}